\documentclass[11pt]{article}
\usepackage{currfile}
\usepackage{color}
\usepackage{mathrsfs}
\usepackage{graphicx}
\usepackage{caption}
\usepackage{subfigure}
\usepackage{amsmath}
\usepackage{amssymb}
\usepackage{amsfonts}
\usepackage{setspace}
\usepackage{epstopdf}
\usepackage{multirow}

\newtheorem{thm}{Theorem}[section]

\newtheorem{cor}[thm]{Corollary}
\newtheorem{lem}[thm]{Lemma}
\newtheorem{prop}[thm]{Proposition}
\newenvironment{proof}{\noindent {\bf
Proof.}}{\rule{3mm}{3mm}\par\medskip}

\def\marker{\>\hbox{${\vcenter{\vbox{
    \hrule height 0.4pt\hbox{\vrule width 0.4pt height 6pt
    \kern6pt\vrule width 0.4pt}\hrule height 0.4pt}}}$}\>}

\def\qed{ \hfill $\square$}
\newcommand{\es}{\chi_S}
\newcommand{\cs}{\chi_{S'}}

\newcommand{\diam}{{\rm diam}}

\DeclareMathOperator{\a13} {{\cal S}_{1, \overline{3}}}
\DeclareMathOperator{\b14}{{\cal S}_{1, \overline{4}}}
\DeclareMathOperator{\c134}{{\cal S}_{1, 3, \overline{4}}}
\DeclareMathOperator{\d133}{{\cal S}_{1, 3, 3}}

\begin{document}

\title{ A characterization of 4-$\chi_S$-vertex-critical graphs for packing sequences with $s_1 =1$ and $s_2\ge 3$}
\author{Sandi Klav\v{z}ar$^{1,2,3,}$\thanks{Email: \texttt{sandi.klavzar@fmf.uni-lj.si}}
\and
Hui Lei$^{4,}$\thanks{Email: \texttt{hlei@nankai.edu.cn}}
\and
Xiaopan Lian$^{5,}$\thanks{Corresponding author; Email: \texttt{xiaopanlian@mail.nankai.edu.cn}}
\and
Yongtang Shi$^{5,}$\thanks{Email: \texttt{shi@nankai.edu.cn}}
}
\maketitle

\begin{center}
	$^1$ Faculty of Mathematics and Physics, University of Ljubljana, Slovenia \\
	\medskip

	$^2$ Faculty of Natural Sciences and Mathematics, University of Maribor, Slovenia\\
	\medskip

	$^3$ Institute of Mathematics, Physics and Mechanics, Ljubljana, Slovenia\\
	\medskip

	$^4$ School of Statistics and Data Science, LPMC and KLMDASR\\
             Nankai University, Tianjin 300071, China\\
	\medskip

    $^5$ Center for Combinatorics and LPMC, Nankai University, Tianjin, China\\
	\medskip
\end{center}

\begin{abstract}
If $S=(s_1,s_2,\ldots)$ is a non-decreasing sequence of positive integers, then the  $S$-packing $k$-coloring of a graph $G$ is a mapping $c: V(G)\rightarrow[k]$ such that if $c(u)=c(v)=i$ for $u\neq v\in V(G)$, then $d_G(u,v)>s_i$. The $S$-packing chromatic number of $G$ is the smallest integer $k$ such that $G$ admits an $S$-packing $k$-coloring. A graph $G$ is $\chi_S$-vertex-critical if $\chi_S(G-u) < \chi_S(G)$ for each $u\in V(G)$. If $G$ is $\chi_S$-vertex-critical and $\chi_S(G) = k$, then $G$ is $k$-$\chi_S$-vertex-critical. In this paper, $4$-$\chi_S$-vertex-critical graphs are characterized for sequences $S = (1,s_2, s_3, \ldots)$ with $s_2 \ge 3$. There are $28$ sporadic examples and two infinite families of such graphs.
\end{abstract}


\medskip\noindent
{\bf AMS Subj.\ Class.\ (2020)}:  05C15, 05C12

\section{Introduction}

A {\it packing $k$-coloring} of a graph $G=(V(G),E(G))$ is a mapping $c:V(G)\rightarrow [k]$ such that if $u\neq v$ and $c(u)=c(v)=i$, then $d_G(u,v)>i$. Here and later, $d_G(u,v)$ denotes the length of a shortest $u,v$-path, and $[k] = \{1,\ldots, k\}$. The {\it packing chromatic number}, $\chi_\rho(G)$, of $G$ is the smallest integer $k$ such that $G$ admits a packing $k$-coloring. This concept was proposed in~\cite{goddard-2008}. The seminal paper was followed by~\cite{BSD}, where the nowadays established name and notation was proposed. The development on the packing chromatic number up to 2020 has been summarized in the substantial survey~\cite{BJS}. Research into this concept is still flourishing, the developments after the survey include~\cite{bidine-2021, bozovic-2021, bresar-2020, deng-2021, fresan-2021}.

A more general concept is the $S$-packing coloring. Let $S=(s_1,s_2,\ldots)$ be a non-decreasing sequence of positive integers; we will refer to $S$ as a {\em packing sequence}. An {\it $S$-packing $k$-coloring} of $G$ is a mapping $c:V(G)\rightarrow [k]$ such that if $u\neq v$ and $c(u)=c(v)=i$, then $d_G(u,v)>s_i$. For example, a $(1,1,1,\ldots)$-packing coloring is the standard proper vertex coloring, and if $S=(1,2,3,\ldots)$, then it is just the packing coloring.  The {\it$S$-packing chromatic number}, $\es(G)$, of $G$ is the smallest integer $k$ such that $G$ admits an $S$-packing $k$-coloring. This concept was introduced by Goddard and Xu~\cite{goddard-2012}; for more results see \cite{bresar-2021, gastineau-2015a, gastineau-2019, goddard-2014, kostochka-2021, liu-2020, yang-2023}.

If $S_1=(s^1_1,s^1_2,\ldots)$ and $S_2=(s^2_1,s^2_2,\ldots)$ are (packing) sequences with $|S_1| = |S_2|$, then $S_2\succeq S_1$ means the coordinate order, that is, $S_2\succeq S_1$ if $s^2_i\geq s^1_i$ for every $i\in [|S_1|]$. If $S_2\succeq S_1$ and $G$ admits an $S_2$-packing $k$-coloring, then $G$ also admits an $S_1$-packing $k$-coloring. In~\cite[Theorem 3.1]{gastineau-2015a}, Gastineau proved the following appealing dichotomy result: If $S$ is a packing sequence with $|S| = 4$, then the decision problem whether a given graph $G$ admits an $S$-packing coloring is polynomial-time solvable if $S\succeq S'$, where $S' \in \{(2, 3, 3,  3), (2, 2, 3, 4), (1, 4, 4, 4), (1, 2, 5, 6)\}$, and NP-complete otherwise.

We have now arrived at the central concept of interest in this paper. A graph $G$ is {\it packing chromatic vertex-critical} if $\chi_\rho(G-u)<\chi_\rho(G)$ holds for each $u\in V(G)$. When $\chi_\rho(G)=k$, we more precisely say that $G$ is $k$-$\chi_\rho$-{\it vertex-critical}. More generally, if $S$ is a packing sequence, then $G$ is {\it $S$-packing chromatic vertex-critical} if $\es(G-u)<\es(G)$ holds for each $u\in V(G)$, and if  $\es(G)=k$, then we say that $G$ is $k$-$\es$-{\it vertex-critical}. We also add that a closely related concept of {\it packing chromatic critical graphs}, where the packing chromatic number strictly decreases on an arbitrary proper subgraph, has been studied in~\cite{bresar-2019++}.

Packing chromatic vertex-critical graphs were introduced in~\cite{klavzar-2019}. Among other results, $3$-$\chi_\rho$-vertex-critical graphs were characterized and a partial characterization of $4$-$\chi_\rho$-vertex-critical graphs was provided. The latter characterization has been completed in~\cite{ferme-2021+}. In \cite{holub-2020}, $3$-$\es$-vertex-critical graphs were characterized for all possible packing sequences, while $4$-$\es$-vertex-critical graphs were characterized for packing sequences ($s_1, s_2, s_3, \ldots)$ with $s_1\ge 2$. In this article we supplement the latter result by characterizing 4-$\es$-vertex-critical graphs for packing sequences with $s_1=1$ and $s_2\ge3$. The result is given in Section~\ref{sec:main}, while in the next section we introduce some additional notation and list known properties of $S$-packing colorings needed here.

\section{Preliminaries}
\label{sec:prelim}

If $G$ is a graph, then we use $n(G)$ to denote its order,  $\diam(G)$ to denote its diameter, and $\chi(G)$ to denote its chromatic number.  For $x\in V(G)$, let $N^i_G(x)$ be the set of vertices which are at distance $i$ from $x$ in $G$. In particular, $N_G(x) = N^1_G(x)$ is the neighborhood of $x$. The degree of $x$ is $d_G(x) = |N_G(x)|$. Let $C_n$, $P_n$, and $K_n$ denote the cycle, the path, and the complete graph on $n$ vertices, respectively. A set $A\subseteq V(G)$ is $k$-{\it independent} if $A$ induces a subgraph that can be properly colored by $k$ colors. Let $\alpha_k(G)$ be the cardinality of a largest $k$-independent set of $G$.

If in a packing sequence the term $i$ repeats $\ell$ times, we may abbreviate the corresponding subsequence by $i^\ell$. For example, if $S = (1,\ldots,1,s_{\ell+1},\ldots)$ (where clearly $1$ appears $\ell$ times), then we may shortly write $S = (1^\ell,s_{\ell+1},\ldots)$. If $\varphi:V(G)\rightarrow [k]$ is an $S$-packing $k$-coloring  of $G$, then $\varphi^{-1}(i)$, $i\in[k]$, is the set of vertices $x$ with $\varphi(x) = i$. We will also use the following convention. Consider the vertex set $V(G) = \{v_1, \ldots, v_n\}$ of $G$ as an ordered set, and let $\varphi$ be an $S$-packing coloring of $G$. Then we will explicitly describe $\varphi$ as follows: $\varphi =  ``\varphi(v_1)~\cdots~\varphi(v_n)"$. Typically, the order of vertices will be alphabetic. For instance, if $V(G) = \{a,b,c,d\}$, and $\varphi(a)=1$, $\varphi(b)=2$, $\varphi(c)=1$, and $\varphi(d)=3$, then $\varphi = ``1~2~1~3"$.

We next recall some known results that will be needed in the rest.

\begin{prop}{\rm \cite{goddard-2008}}
\label{spc7}
Let $n\ge 3$. If $n=3$ or $n=4k$, $k\ge 1$, then $\chi_\rho(C_n)= 3$; otherwise $\chi_\rho(C_n)=4$.
\end{prop}

\begin{lem}{\rm \cite{goddard-2012}}\label{lem:subgraph}
\label{spc4}
If $S$ is a packing sequence and $H$ is a subgraph of $G$, then $\es(H)\leq\es(G)$.
\end{lem}

\begin{prop}{\rm \cite{goddard-2012}}
\label{spc6}
Let $S = (1^\ell, s_{\ell+1},\ldots)$, where $\ell\ge1$ and $s_{\ell+1}\ge2$, and let $G$ be a graph. Then $\es(G) \le n(G)-\alpha_\ell(G) +\min\{\ell, \chi(G)\}$ with equality if and only if $\diam(G) \le s_{\ell+1}$.
\end{prop}

\begin{lem}{\rm \cite{klavzar-2019}}
\label{spc5}
If $S$ is a packing sequence and $G$ is a $\es$-vertex-critical graph, then $G$ is connected.
\end{lem}

Finally, the following notation will be useful. Suppose we wish to consider all the packing sequences $S = (s_1, s_2, s_3, \ldots)$, for which $s_1 = 2$, $s_2 \ge 4$, and $s_3 = 5$ hold. We will denote the set of all such packing sequences by ${\cal S}_{2,\overline{4},5}$, that is,
$${\cal S}_{2,\overline{4},5} = \{(s_1,s_2,s_3,\ldots):\ s_1 = 2, s_2 \ge 4, s_3 = 5\}\,.$$
Note that since ${\cal S}_{2,\overline{4},5}$ is a set of packing sequences, we have $s_2\in \{4,5\}$ when $S\in {\cal S}_{2,\overline{4},5}$. The general notation should be clear from this example. For instance, using this notation we can state that $S \succeq (s_1,s_2,s_3,\ldots)$ if and only $S\in {\cal S}_{\overline{s}_1,\overline{s}_2, \overline{s}_3, \ldots}$.

%%%%%%%%%%%%%%%%%%%%%%%%%%%%%%%%%%%%%%%%%%%%
\section{Vertex-critical graphs for different packing sequences}
\label{sec:main}
%%%%%%%%%%%%%%%%%%%%%%%%%%%%%%%%%%%%%%%%%%%%

As mentioned in the introduction, a characterization of $3$-$\es$-vertex-critical graphs is known for all possible packing sequences, while $4$-$\es$-vertex-critical graphs were by now characterized for packing sequences from ${\cal S}_{\overline{2}}$. In this section we supplement the latter result by characterizing 4-$\es$-vertex-critical graphs for packing sequences $S$ from ${\cal S}_{1, \overline{3}}$. To this end note that
$$\a13 = \b14 \cup \c134 \cup \d133\,.$$
In view of this fact we will solve our problem by characterizing 4-$\es$-vertex-critical graphs for packing sequences from each of the sets $\b14$, $\c134$, and $\d133$.

In Figs.~\ref{spcfig1} and~\ref{spcfig2}, several graphs are drawn that will turn out to be 4-$\es$-vertex-critical for packing sequences from ${\cal S}_{1, \overline{3}}$. Fig.~\ref{spcfig1} contains two small families of graphs, the family $\mathcal{C}_5$ contains four graphs of order $5$, while $\mathcal{C}_6$ contains three graphs of order $6$. Fig.~\ref{spcfig2} displays the family of graphs ${\cal H}$ consisting of graphs $H_i$, $i\in [15]$.

\begin{figure}[ht!]
  \begin{center}
    \includegraphics[width=8cm]{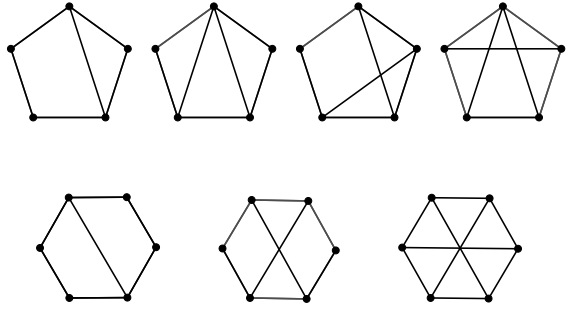}\\
  \caption{Family $\mathcal{C}_5$ (top row) and family $\mathcal{C}_6$ (bottom row)}
  \label{spcfig1}
  \end{center}
\end{figure}

\begin{figure}[ht!]
  \begin{center}
    \includegraphics[width=14cm]{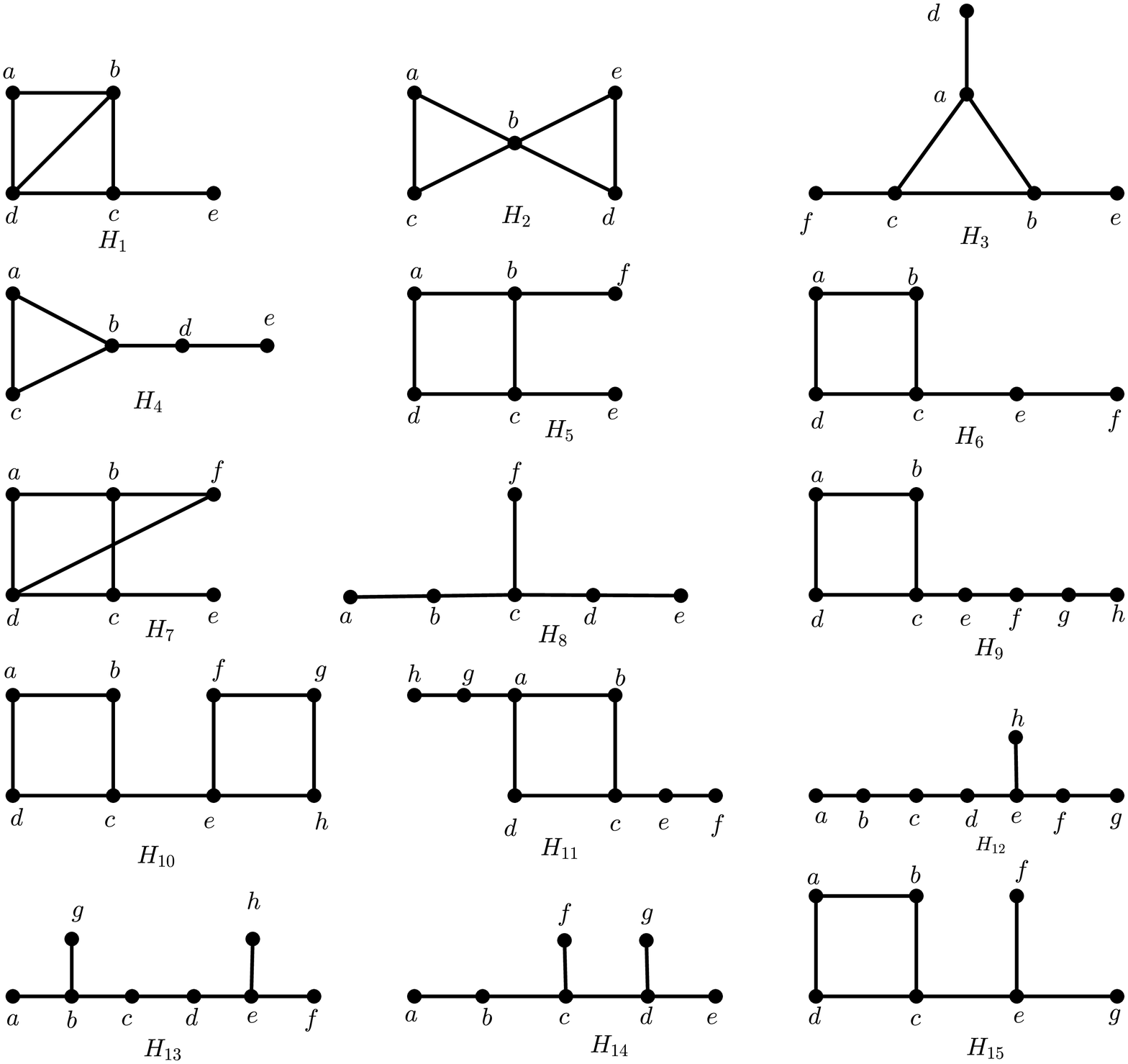}\\
  \caption{Family ${\cal H} = \{H_i:\ i\in [15]\}$}
  \label{spcfig2}
  \end{center}
\end{figure}

In the rest we will frequently consider different subsets of ${\cal H}$. To shorten the presentation, we will specify subsets of ${\cal H}$ by (ranges of) indices. For instance, ${\cal H}_{1-3,7,9-11} = \{H_1, H_2, H_3, H_7, H_9, H_{10}, H_{11}\}$.

First we detect the following critical graphs.

\begin{lem}\label{KCH}
Let $S\in \a13$. Then each of the graphs from
$\mathcal{G} = \{K_4, C_5, C_6\} \cup
\mathcal{C}_5 \cup
\mathcal{C}_6 \cup
{\cal H}_{1-5,7}$
is $4$-$\es$-vertex-critical.
\end{lem}

\begin{proof}
%Since $s_2\ge 3$,
Observe that for each $G\in \mathcal{G}$, $\diam(G)\leq s_2$. Using Proposition~\ref{spc6} it is then straightforward to check that $\es(G)=4$ for each $G\in \mathcal{G}$. It remains to show that each graph $G\in \mathcal{G}$ is 4-$\es$-vertex-critical.

By Proposition~\ref{spc6}, we have $\es(K_3)=3-1+1=3$, $\es(P_k)\le k-\alpha(P_k)+1\le3$ for $k\le5$, $\es(G-x)=4-2+1=3$ for any $G\in \mathcal{C}_5$ and $x\in V(G)$, and $\es(G-x)=5-3+1=3$ for any $G\in \mathcal{C}_6$ and $x\in V(G)$. Therefore, $K_4$, $C_5$, $C_6$, each of the graphs from $\mathcal{C}_5$,  and each of the graphs from $\mathcal{C}_6$ are $\es$-vertex-critical.

Now we prove that each graph in ${\cal H}_{1-5,7}$ is 4-$\chi_S$-vertex-critical where $S\in  \a13$. First consider the case where $S\in \b14$.  We give an $S$-packing 3-colorings $\varphi$ for every $G-x$, where $G\in {\cal H}_{1-5,7}$ and $x\in V(G)$. (By symmetry, we do not need to consider all the vertices.) Suppose $G=H_1$. Then we define $\varphi$ as ``1 2 3 1",``1 1 3 2",``1 2 3 2",``1 2 1 3" when $x=a,b,c,e$, respectively. Suppose $G=H_2$. Then we define $\varphi$ as ``2 1 1 3", when $x=a$ or $x=b$. Suppose $G=H_3$. Then we define $\varphi$ as ``2 3 1 1 1", ``1 2 3 1 1" when $x=a,d$, respectively. Suppose $G=H_4$. Then we define $\varphi$ as ``3 1 1 2", ``2 1 1 2", ``2 1 3 2", ``2 3 1 1" when $x=a,b,d,e$, respectively. Suppose $G=H_5$. Then we define $\varphi$ as ``1 2 1 1 3", ``3 2 1 1 3", ``3 1 2 1 1" when $x=a,b,f$, respectively. Suppose $G=H_7$. Then we define $\varphi$ as ``1 2 1 1 3", ``1 2 3 1 1", ``2 1 1 1 3", ``1 2 1 3 1", when $x=a,b,c,e$, respectively.  We have thus verified that each $G\in {\cal H}_{1-5,7}$ is $\chi_{S}$-vertex-critical for $S\in \b14$.

Finally suppose $S\in \c134 \cup \d133$. Let $G\in {\cal H}_{1-5,7}$ and $x\in V(G)$ be an arbitrary vertex. Since $\chi_S(G)=4$, it suffices to show that $\chi_S(G-x)=3$.  Observe that for any packing sequence $S\in \c134 \cup \d133$ there is a packing sequence $S'\in \b14$ such that $S' \succeq S$. Thus, the above $S'$-packing $3$-coloring of $G - x$, where $G\in {\cal H}_{1-5,7}$ and $x\in V(G)$, yields an $S$-packing $3$-coloring of $G - x$. Therefore, we are finished.
\end{proof}

%%%%%%%%%%%%%%%%%%%%%%%%%%%%%%%%%%%%%%%%
\subsection{$4$-$\es$-vertex-critical graphs for $S\in \b14 \cup \c134$}
%%%%%%%%%%%%%%%%%%%%%%%%%%%%%%%%%%%%%%%%

In this subsection we characterize 4-$\es$-vertex-critical graphs for $S\in \b14$ and for $S\in \c134$. The results are given in Theorems~\ref{spcthm1}, \ref{spcthm2}, and \ref{spca}.

\begin{lem}\label{HH68}
$P_6$, $H_6$, and $H_8$ are $4$-$\es$-vertex-critical graphs for $S\in \b14$.
\end{lem}

\begin{proof}
Let $S\in \b14$.

First we prove that $P_6$ is 4-$\chi_S$-vertex-critical.  Suppose that $P_6=abcdef$ has an $S$-packing $3$-coloring $\varphi$. Since $|\varphi^{-1}(1)|\leq 3$, we have $|\varphi^{-1}(2)|\geq 2$  or  $|\varphi^{-1}(3)|\geq 2$. Since $s_2\ge 4$, we must have $\varphi(a)=\varphi(f)=\alpha\in\{2,3\}$. Then at least three vertices of $\{b,c,d,e\}$ must receive color $1$, but this is impossible. The pattern $``1~2~1~3~1~4"$ gives an $S$-packing $4$-coloring of $P_6$, so $\es(P_6)=4$. By Proposition~\ref{spc6}, $\es(P_k)\le k-\alpha(P_k)+1\le3$ for $k\le5$. Hence, $P_6$ is 4-$\es$-vertex-critical.

Now we prove that both $H_6$ and $H_8$ are 4-$\chi_S$-vertex-critical. Observe that $ \alpha(H_6)=\alpha(H_8)=3$. By Proposition~\ref{spc6}, we have $\es(H_6) = \es(H_8) = 6-3+1=4$. Now we give  an $S$-packing $3$-coloring $\phi$ of $G-x$ for $G\in \{H_6, H_8\}$ and $x\in V(G)$. If $G= H_6$, then we define $\phi$ as ``1 2 1 1 3", ``1 1 2 3 1", ``2 1 1 1 3", ``3 1 2 1 3", ``3 1 2 1 1" when $x=a,b,c,e,f$, respectively. If $G=H_8$, then we define $\phi$ as ``1 2 1 3 1", ``2 2 1 3 1", ``2 1 1 3 1", ``1 2 1 3 1" when $x=a,b,c,f$, respectively.
\end{proof}

\begin{lem}\label{spc2}
Each of the graphs from $\{P_8, C_8\} \cup {\cal H}_{9,11-15}$ is $4$-$\es$-vertex-critical for $S\in \c134$.
\end{lem}

\begin{proof}
Let $S\in \c134$.

Suppose that $P_8=abcdefgh$ has an $S$-packing $3$-coloring $\varphi$.
Since $|\varphi^{-1}(1)| \le 4$, $|\varphi^{-1}(2)| \le 2$, and $|\varphi^{-1}(3)| \le 2$, we have $|\varphi^{-1}(1)|=4$, $|\varphi^{-1}(2)|=2$ and $|\varphi^{-1}(3)|=2$. Without loss of generality assume $\varphi(a)=\varphi(c)=\varphi(e)=\varphi(g)=1$. Then we have $c(b)=c(h)=3$ because $s_3\ge4$. Thus $d$ and $f$ must receive color $2$, a contradiction. The pattern $``1~2~1~3~1~2~1~4"$ gives an $S$-packing $4$-coloring of $P_8$, so $\es(P_8)=4$. Since $\es(P_8)=4$ and the pattern $``1~2~1~3~1~2~1~4"$ gives an $S$-packing 4-coloring of $C_8$, we have $\es(C_8)=4$. The first $k$ entries in the pattern $``1~2~1~3~1~2~1"$ gives an $S$-packing 3-coloring of $P_k$ with $k\le7$, so $P_8$ and $C_8$ are 4-$\es$-vertex-critical.

If $G= H_9$, then the pattern ``2 1 3 1 1 2 1 4" is an $S$-packing 4-coloring of $H_9$, so $\es(H_9) \le 4$. Suppose that $H_9$ has an $S$-packing 3-coloring $\varphi$. Observe that $\{\varphi(a),\varphi(c)\}=\{2,3\}$, which implies that $\varphi(e)=\varphi(f)=1$ or $\varphi(g)=\varphi(h)=1$, a contradiction. Hence $\es(H_9)=4$. Now we give an $S$-packing 3-coloring $\phi$ of $H_9-x$ for any $x\in V(H_9)$. We define $\phi$ as ``1 2 1 1 3 1 2", ``2 1 1 1 2 1 3", ``1~2~1~3~1~2~1", ``2 1 3 1 2 1 3", ``2 1 3 1 1 1 3", ``2 1 3 1 1 2 3", ``2 1 3 1 1 2 1" when $x=a,c,d,e,f,g,h$, respectively.

If $G=H_{11}$, then the pattern ``2 1 3 1 1 2 1 4"  is an $S$-packing 4-coloring, so $\es(H_{11})\le 4$. Suppose that $H_{11}$ admits an $S$-packing 3-coloring $\varphi$. Observe that $\{\varphi(a),\varphi(c)\}=\{2,3\}$.  Then $\varphi(g)=\varphi(h)=1$ or $\varphi(e)=\varphi(f)=1$, a contradiction. Hence $\es(H_{11})=4$. Now we give an $S$-packing 3-coloring $\phi$ of $H_{11}-x$ for any $x\in V(H_{11})$. We define $\phi$ as ``1 3 1 1 2 1 3", ``1~3~1~2~1~2~1", ``2 1 3 1 1 2 3", ``2 1 3 1 1 2 1" when $x=a,d,g,h$, respectively.

If $G=H_{12}$, then the pattern ``4 1 2 1 3 1 2 1"  is an $S$-packing 4-coloring of $H_{12}$. Hence $\es(H_{12})\le4$. Suppose $H_{12}$ admits an $S$-packing 3-coloring $\varphi$, then $\varphi(e)=2$ or $3$. If $\varphi(e)=2$, then we have $\{\varphi(f),\varphi(g)\}=\{1,3\}$ and  $\varphi(d)=1$. Thus $\varphi(c)\in\{2,3\}$, a contradiction. If $\varphi(e)=3$, then $\varphi(d)=1$, $\varphi(c)=2$, $\varphi(b)=1$. Thus $\varphi(a)\in\{2,3\}$, a contradiction. Therefore $\es(H_{12})=4$. Now we give an $S$-packing 3-coloring $\phi$ of $H_{12}-x$ for any $x\in V(H_{12})$. We define $\phi$ as ``1 2 1 3 1 2 1", ``3 2 1 3 1 2 1", ``3 1 1 3 1 2 1", ``3 1 2 3 1 2 1", ``3 1 2 1 1 2 1", ``2 1 3 1 2 2 1", ``2 1 3 1 2 1 1", ``1~2~1~3~1~2~1" when $x=a,b,c,d,e,f,g,h$, respectively.

If $G=H_{13}$, then the pattern ``1 2 1 3 1 2 1 4" is an $S$-packing 4-coloring. Hence $\es(H_{13}) \le 4$. Suppose that $H_{13}$ admits an $S$-packing 3-coloring $\varphi$, then  $\{\varphi(b),\varphi(e)\}=\{2,3\}$. Then $\varphi(c)=\varphi(d)=1$, a contradiction. Therefore $\es(H_{13})=4$. Now we give an $S$-packing 3-coloring $\phi$ of $H_{13}-x$ for any $x\in V(H_{13})$. We define $\phi$ as ``1 3 1 2 1 1 1",``1 3 1 2 1 1 1", ``2 1 3 1 2 1 1" when $x=b,c,g$, respectively.

If $G=H_{14}$, then the pattern ``2 1 3 1 2 1 4" is an $S$-packing 4-coloring of $H_{14}$. Hence $\es(H_{14})\le 4$. Suppose that  $H_{14}$ admits an $S$-packing 3-coloring $\varphi$, then $\{\varphi(c),\varphi(d)\}=\{2,3\}$. Thus $\varphi(a)=\varphi(c)>1$ or $\varphi(a)=\varphi(d)>1$, a contradiction. Therefore $\es(H_{14})=4$. Now we give an $S$-packing 3-coloring $\phi$ of $H_{14}-x$ for any $x\in V(H_{14})$. We define $\phi$ as ``1 2 3 1 1 1",``2 2 3 1 1 1", ``1 2 3 1 1 1", ``2 1 3 2 1 1",``1 2 1 3 1 1",``2 1 3 1 2 1", when $x=a,b,c,d,f,g$, respectively.

Finally, if $G=H_{15}$, then the pattern ``4 1 2 1 3 1 1" is an $S$-packing 4-coloring of $H_{15}$. Hence $\es(H_{15})\le4$. Suppose that  $H_{15}$ admits an $S$-packing 3-coloring $\varphi$, then $\{\varphi(c),\varphi(e)\}=\{2,3\}$ and $\varphi(b)=\varphi(d)=1$. Thus $\varphi(a)\in\{2,3\}$, a contradiction.  Now we give an $S$-packing 3-coloring $\phi$ of $H_{15}-x$ for any $x\in V(H_{15})$. We define $\phi$ as ``1 2 1 3 1 1", ``1 1 2 3 1 1 ", ``2 1 1 3 1 1", ``1 3 1 2 1 1", ``2 1 3 1 1 2", when $x=a,b,c,e,f$, respectively.\end{proof}

\begin{lem}\label{n34}
If $S\in \b14 \cup \c134$, $G$ is a $4$-$\es$-vertex-critical graph with at least one cycle, and $C$ is a longest cycle of $G$, then the following hold.
\begin{itemize}
\item[(a)]If $n(C) = 3$, then $G\in {\cal H}_{2-4}$.
\item[(b)]If $n(C) = 4$ and $C$ contains a chord, then $G\in \{K_4, H_1\}$.
\item[(c)]If $n(C) \in \{5, 6\}$, then $G\in \{C_{n(C)}\}\cup \mathcal{C}_{n(C)}$.
\end{itemize}
\end{lem}

\begin{proof}
Let $S\in \b14 \cup \c134$. Note that the graphs from Lemma~\ref{KCH} are 4-$\es$-vertex-critical. Let now $G$ be a $4$-$\es$-vertex-critical graph with a longest cycle $C$.

(a) Suppose $n(C)=3$. Let $V(C)=\{a,b,c\}$. We first assume that $G$ contains only one triangle. If $H_3$ or $H_4$ is a subgraph of $G$, then we actually have $G = H_3$ or $G = H_4$, for otherwise we find another triangle in $G$ or a cycle longer than $3$. If $d_{G}(v)\geq 3$ holds for each vertex of $C$, then $G = H_3$ since $H_3$ is $4$-$\es$-vertex-critical.  If $d_{G}(v)=2$ for some $v\in\{a,b,c\}$, then assume without loss of generality that $d_{G}(a)=2$. If $N^2_{G}(b)\setminus N_G(c)=\emptyset$ and $N^2_{G}(c)\setminus N_G(b)=\emptyset$, then $V(G)\setminus\{b,c\}$ is an independent set in $G$, and so a coloring $\varphi$ with $\varphi(b)=2$, $\varphi(c)=3$ and other vertices with color $1$ is an $S$-packing 3-coloring of $G$, a contradiction. So $N^2_{G}(b)\setminus N_G(c)\neq \emptyset$ or $N^2_{G}(c)\setminus N_G(b)\neq\emptyset$. Since $H_4$ is 4-$\es$-vertex-critical,   $G= H_4$.

Suppose secondly that there are at least two triangles in $G$. Since $H_4$ is $\es$-vertex-critical, the triangles in $G$ have exactly one common vertex, for otherwise there is vertex $v$ in $G$ such that $H_4\subseteq G-v$ or $n(C)\ge 4$. This implies that $H_2$ is a spanning subgraph of $G$. Since $n(C)=3$, we conclude that $G= H_2$.

(b) Suppose $n(C)=4$. Let $C = abcda$. If $ac\in E(G)$ and $bd\in E(G)$, then
$G= K_4$ by Lemma~\ref{KCH}. Suppose $bd\in E(G)$. If there is a vertex $x\in N_G(b)\setminus V(C)$ such that $N_G(x)\setminus V(C)\neq\emptyset$, then $H_4\subseteq G-a$, a contradiction. Therefore, for any vertex $x\in (N_G(b)\cup N_G(d))\setminus V(C)$ we have $N_G(x)\setminus V(C)=\emptyset$. If $d_G(a)=d_G(c)=2$, then $V(G)\setminus\{b,d\}$ is an independent set in $G$, and so a mapping $\varphi$ with $\varphi(b)=2$, $\varphi(d)=3$ and $\varphi(N_G(b)\cup N_G(d)\setminus\{b,d\})=1$ is an $S$-packing 3-coloring of $G$, a contradiction. Thus $d_G(a)\ge2$ or $d_G(c)\ge 2$.  It implies that $H_1$ is a subgraph of $G$. Since $n(C)=4$ and by Lemma~\ref{KCH} $H_1$ is 4-$\chi_S$-vertex-critical, we have $G= H_1$.

(c) Suppose finally that $n(C)\in\{5,6\}$. Since $C$ is 4-$\es$-vertex-critical, $C$ is a spanning subgraph of $G$. If $n(C)=5$, then since $K_4$ and all the four graphs from $\mathcal{C}_5$ are 4-$\es$-vertex-critical,  $\mathcal{C}_5$ is the family of $4$-$\es$-vertex-critical graphs that contain $C_5$ as a proper spanning subgraph. If $n(C)=6$, then since $C_5$ is 4-$\es$-vertex-critical, any two vertex at distance $2$ are not adjacent in $C_6$. Hence $\mathcal{C}_6$ is the family of $4$-$\es$-vertex-critical graphs that contain $C_6$ as a proper spanning subgraph by Lemma~\ref{KCH}.  Therefore if $n(C)\in \{5,6\}$ and $G$ is $4$-$\es$-vertex-critical, then $G\in \{C_{n(C)}\}\cup \mathcal{C}_{n(C)}$.
\end{proof}

We can now state out first characterization.

\begin{thm}\label{spcthm1}
Let $S\in \b14$. Then a graph $G$ is $4$-$\es$-vertex-critical if and only if
$$G\in \{K_4, C_5, C_6, P_6\}\cup \mathcal{C}_5\cup \mathcal{C}_6 \cup {\cal H}_{1-8}\,.$$
\end{thm}

\begin{proof}
Let $S\in \b14$ and let $G$ be 4-$\es$-vertex-critical. First suppose that $G$ contains a cycle, and let $C$ be a longest cycle of $G$. Since $P_6$ is $4$-$\es$-vertex-critical, Lemma~\ref{HH68} implies $n(C) \le 6$. By Lemma~\ref{n34} and the fact that $\es(C_4)=3$, it remains to consider the case in which $n(C)=4$, $n(G) \ge 5$, and there is no chord in $C$. Let $C = abcda$. Since $\chi_S(C)\le 3$, there is a vertex $w\in V(C)$ such that $N_G(w)\setminus V(C)\neq \emptyset$.  Let $w_1\in N_G(w)\setminus V(C)$.

First suppose $N_G(w_1)\setminus V(C)\neq \emptyset$. We may assume that $w=c$ and $w_1=e$. Let $f\in N_G(e)\setminus V(C)$. Then $H_6$ is subgraph of $G$. By Lemma~\ref{HH68}, $H_6$ is a spanning subgraph of $G$.
%Suppose that there is a vertex $w\in V(C)$ such that $N^2_{G}(w)\setminus V(C) \neq \emptyset$. We may assume that $w=c$.
%It is equivalent to consider that $H_6$ is a spanning subgraph of $G$ by Lemma~\ref{HH68}.
Since $G$ is $C_k$-free for $k\ge 5$, at most one of the edges $\{ae,cf\}$ can be possibly contained in $G$. If $ae\notin E(G)$ and $cf\notin E(G)$, then $G= H_6$ by Lemma~\ref{HH68}. If $ae\in E(G)$, then $G= H_7$ by Lemma~\ref{KCH}. If $cf\in E(G)$, then $H_4\subseteq G-b$, a contradiction.

%Suppose next that $N^2_{G}(w)\setminus V(C)=\emptyset$ for each $w\in V(C)$.
Thus we may assume that $N_G(w_1)\setminus V(C)=\emptyset$ for each $w_1\in N_G(w)\setminus V(C)$. It implies that $N_{G}(u)\setminus V(C)$ is an independent set for any $u\in V(C)$. %, for otherwise there is a vertex $w'\in V(C)$ such that $H_4\subseteq G-w'$, a contradiction.
If $N_G(b)\cup N_G(d)\setminus V(C)=\emptyset$, then $V(G)\setminus \{a,c\}=N(a)\cup N(c)$ is an independent set in $G$, and so a mapping  $\varphi$ with $\varphi(a)=2$, $\varphi(c)=3$ and $\varphi(N(a)\cup N(c))=1$ is an $S$-packing 3-coloring of $G$, a contradiction. Thus $N_G(b)\cup N_G(d)\setminus V(C)\ne\emptyset$ and $N_G(a)\cup N_G(c)\setminus V(C)\ne\emptyset$, and so $H_5$ is a spanning subgraph of $G$ by Lemma~\ref{KCH}. If some edge from $\{af,de\}$ or from $\{ef,cf,be\}$ is contained in $G$, then $H_4\subseteq G-y$ for some $y\in V(G)$ or $C_k\subseteq G$ with $k\ge5$, a contradiction. Therefore, only one of $df$ and $ae$ can be contained in $G$, and so $G\in {\cal H}_{5,7}$ by Lemma~\ref{KCH}.

Suppose now that $G$ is acyclic. If $P$ is a longest path in $G$, then $n(P) \le 6$ by Lemma~\ref{HH68}. If $n(P) = 6$, then $G = P_6$. If $n(P) = 5$, then let $P_5=abcde$. If $d_G(c)=2$, then the mapping $\varphi$ with $\varphi(b)=2$, $\varphi(d)=3$ and $\varphi(N_G(b)\cup N_G(d))=1$ is an $S$-packing 3-coloring of $G$ which implies that $\es(G) \le 3$, a contradiction. Therefore $d_G(c)\ge3$. But then $G = H_8$ by Lemma~\ref{HH68}. If $n(P) \le 4$, then we have that $\es(G) \le 3$, so we get no new graph.
\end{proof}

\begin{thm}\label{spcthm2}
Let $S\in \c134$ and let $G$ be a graph with a cycle. Then $G$ is $4$-$\es$-vertex-critical if and only if
$$G\in \{K_4, C_5, C_6, C_8\} \cup \mathcal{C}_5\cup \mathcal{C}_6\cup  {\cal H}_{1-5,7,9,11,15}\,.$$
\end{thm}

\begin{proof}
Let $S\in \c134$ and let $C$ be a longest cycle of $G$. Since $P_8$ is $4$-$\es$-vertex-critical, $n\le 8$. If $n(C)=8$, then  $C$ is a spanning subgraph of $G$ by Lemma~\ref{spc2}. Since $\es(C_5)=\es(C_6)=4$, and $\es(C_7)=7-\alpha(C_7)+1>4$ by Proposition~\ref{spc6}, there is no chord in $C$. Therefore $G = C$ when $n(C)=8$. Since  $\es(C_7)=5$, we have $n(C)\neq 7$. By Lemma~\ref{n34} and the fact $\es(C_4)=3$ it remains to consider the case that $n(C) = 4$, $n(G)\ge5$, and there is no chord in $C$.

Let $C= abcda$. First suppose that there is an edge in $E(C)$, say $bc$, such that $d_G(b)\ge 3$ and $d_G(c)\ge 3$. Then $N_G(b)\cap N_G(c)=\emptyset$ for otherwise $G$ has a cycle of length at least 5. It follows that %For an edge $st\in E(C)$, if  $d_G(s)\ge3$ and $d_G(t)\ge3$, then
$H_5$ is a spanning subgraph of $G$ by Lemma \ref{KCH} because there is no chord in $C$. If $af$ or $de$ is contained in $G$, then $H_4\subseteq G-y$ for some $y\in V(G)$. Hence at most one of $df$ and $ae$ can be added to $G$. Therefore $G\in {\cal H}_{5,7}$ by Lemma \ref{KCH}.

Now consider the case in which $d_G(s)=2$ or $d_G(t)=2$ for each edge $st\in E(C)$. Without loss of generality, suppose that $d_G(b)=2$ and $d_G(d)=2$.
Let $P$ be a longest path with endpoint $c$, such that $a,b,d\notin V(P)$, and let $P'$ be a longest path with endpoint $a$, such that $c,b,d\notin V(P')$. Without loss of generality assume that $n(P) \ge n(P')$. If $n(P)\ge 3$ and $V(P)\cap V(P')\neq \emptyset$, then $G = H_7$. Indeed, for otherwise by the definition of $P$ and $P'$ we have $a,c\not \in V(P)\cap V(P')$, and then for some $k\ge 5$ we have $C_k\subseteq G-b$, a contradiction.  In the rest of the proof we may thus assume that if $n(P) \ge 3$, then $V(P)\cap V(P')= \emptyset$. Since $P_8$ is $4$-$\es$-vertex-critical, $n(P) + n(P') \le 6$.

\medskip\noindent
{\bf Claim.} If $n(P) \le 4$ and $n(P') \le 2$, then $G = H_{15}$.

\medskip\noindent
{\bf Proof.} Since  $n(P') \le 2$, we infer that if $x\in N_G(a)\setminus N_G(c)$ and $y\in N_G(a)\cap N_G(c)$, then $d_G(x)=1$ and $d_G(y)=2$. Hence  $N_G(a)$ is an independent set in $G$. If there is a vertex $x\in N_G(c)\setminus N_G(a)$ such that $d_G(x)\ge3$, then $H_{15}\subseteq G$. Since $H_{15}$ is $4$-$\es$-vertex-critical by Lemma~\ref{spc2}, $H_{15}$ is a spanning subgraph of $G$. Since $n(P') \le 2$ and $d_G(b)=d_G(d)=2$, only edges from $\{fg,cf\}$ are possibly contained in $G$.
If an edge from $fg$ or $cf$ is contained in $G$, then there is a vertex $v\in G$ such that $H_4\subseteq G-v$, a contradiction. Hence $G= H_{15}$. It remains to consider the case in which $d_G(x) \le 2$ holds for each $x\in N_G(c)$. Then $N_G(c)$ is an independent set in $G$, for otherwise $H_4\subseteq G-b$, a contradiction.
Since $n(P) \le 4$ and $d_G(x) \le 2$ for every $x\in N_G(c)$, the second neighborhood $N_G^2(c)$ is an independent set and $d_G(y)=1$ for each $y\in N_G^3(c)$. (It is possible that $N_G^3(c) = \emptyset$.) Then a mapping $\varphi$ with $\varphi(c)=3$, $\varphi(N_G(c)\cup N^3_G(c))=1$, and $\varphi(N^2_G(c))=2$ is an $S$-packing 3-coloring of $G$. This contradiction proves the claim. \qed

\medskip
It remains to consider the following two cases: (i) $n(P) = 5$, $n(P') = 1$, and (ii) $n(P) = n(P') = 3$. If $n(P) = 5$, then
$H_9$ is a spanning subgraph of $G$. (The vertices of $H_9$ are denoted as in Fig.~\ref{spcfig2}.) If some edge from $\{cf,eg,fh\}$ or $ch$ or $cg$ is contained in $G$, then $H_4\subseteq G-b$ or $C_5\subseteq G-b$ or $H_5\subseteq G-b$, respectively, a contradiction. Now we only need to check that whether $eh$ can be added to $H_9$. The graph obtained from $H_9$ by adding the edge $eh$ is $H_{10}$, cf.~Fig.~\ref{spcfig2} again. Then $H_{15}\subseteq H_{10}-g$, a contradiction. Hence  $G= H_9$. If $n(P) = 3$ and $n(P') = 3$, then $H_{11}\subseteq G$. By symmetry, if some edge from $\{af,ge,gf\}$ or $ah$ or $ae$ is contained in $G$, then $C_k\subseteq G-b$ with $k\ge5$ or $H_4\subseteq G-b$ or $H_5\subseteq G-b$, respectively, a contradiction. Since $H_{11}$ is 4-$\es$-vertex-critical, no additional edge can be added to $H_{11}$. We conclude that $G= H_{11}$.
\end{proof}

It remains to consider acyclic graphs for $S\in \c134$.

\begin{thm}\label{spca}
Let $S\in \c134$ and let $G$ be an acyclic graph. Then $G$ is $4$-$\es$-vertex-critical if and only if $G\in \{P_8\} \cup {\cal H}_{12-14}$.
\end{thm}

\begin{proof}
Let $G$ be 4-$\es$-vertex-critical and acyclic. Denote by $P$ a longest path in $G$. If $n(P) = 8$, then Lemma~\ref{spc2} implies that $G = P_8$. Since $\es(G)\leq 3$ when $n(P)\leq 4$, it remains to consider the cases $5\leq n(P)\leq 7$.

Suppose $n(P) = 5$ and let $P=abcde$. If $d_G(c)=2$, then a coloring $\varphi$ with $\varphi(\{c\}\cup N^2_G(c))=1$, $\varphi(b)=2$, and $\varphi(d)=3$ is an $S$-packing 3-coloring of $G$, contradicting the fact that $\es(G)=4$, hence $d_G(c)\ge3$. If $d_G(x)\le2$ for any $x\in N_G(c)$, then the coloring $\varphi$ with $\varphi(N_G(c))=1$, $\varphi(N^2_G(c))=2$, and $\varphi(c)=3$ is an $S$-packing 3-coloring of $G$, a contradiction. So $G = H_{14}$ by Lemma~\ref{spc2}.

Suppose $n(P) = 6$ and let $P=abcdef$. Then either $d_G(s)=2$ or $d_G(t)=2$ for $st\in E(P)\setminus \{ab,ef\}$, otherwise there is a vertex $x\in V(G)$ such that $H_{14}\subseteq G-x$. If $d_G(c)\ge3$, then a mapping $\varphi$ with $\varphi(N_G(c)\cup N^3_G(c))=1$, $\varphi(N^2_G(c))=2$, and $\varphi(c)=3$ is an $S$-packing 3-coloring of $G$, a contradiction. Thus $d_G(c)=d_G(d)=2$. If $d_G(b)=2$, then a mapping $\varphi$ with $\varphi(N_G(e)\cup N^3_G(e))=1$, $\varphi(a)=\varphi(e)=2$, and $\varphi(c)=3$ is an $S$-packing 3-coloring of $G$. Thus $d_G(b)\geq3$ and $d_G(e)\geq3$. Hence $G= H_{13}$ by Lemma~\ref{spc2}.

Let finally $P=abcdefg$. If $d_G(x)=2$ for any $x\in N_G(d)$, a mapping $\varphi$ with $\varphi(N_G(d)\cup N^3_G(d))=1$, $\varphi(N^2_G(d))=2$, and $\varphi(d)=3$ is an $S$-packing 3-coloring of $G$ contradicting the fact $\es(G)=4$.  Hence $G = H_{12}$ by Lemma~\ref{spc2}.
\end{proof}

Combining Theorem~\ref{spca} with Theorem~\ref{spcthm2} we get:

\begin{cor}\label{cor:S=134}
Let $S\in \c134$ and let $G$ be a graph. Then $G$ is $4$-$\es$-vertex-critical if and only if
$$G\in \{K_4, C_5, C_6, C_8, P_8\} \cup \mathcal{C}_5\cup \mathcal{C}_6\cup  {\cal H}_{1-5,7,9,11-15}\,.$$
\end{cor}

%%%%%%%%%%%%%%%%%%%%%%%%%%%%%%%%
\subsection{$4$-$\es$-vertex-critical graphs for $S\in \d133$}
%%%%%%%%%%%%%%%%%%%%%%%%%%%%%%%%

In this subsection we consider packing sequences $S\in \d133$. In Lemmas~\ref{pcn}, \ref{spc3}, and \ref{HH145}, we present some graphs that are 4-$\es$-vertex-critical. After that we characterize $4$-$\es$-vertex-critical graphs by distinguishing the distance between vertices of degree at least $3$.

\begin{lem}\label{pcn}
If $S\in \d133$, then the following hold.
\begin{itemize}
\item[(a)] If $n\geq 4$, then $\es(P_n)=3$.
\item[(b)] Let $n\ge3$. If $n=3$ or $n\equiv0\bmod{4}$, then $\es(C_n)= 3$. If $n\equiv1,2\bmod{4}$, or $n\equiv3\bmod{4}$ and $s_4<\lfloor n/2 \rfloor$, then $\es(C_n)=4$; otherwise, $\es(C_n)=5$. Moreover, $C_n$ is $\es$-vertex-critical when $n\not\equiv 0\bmod 4$ and $n\ge 5$.
\end{itemize}
\end{lem}

\begin{proof}
(a) Note that $\es(P_n)\geq 3$ for $n\geq 4$. The pattern $``1~2~1~3~1~2~1~3\ldots"$  is an  $S$-packing 3-coloring of $P_n$. Thus $\es(P_n)=3$ for $n\geq 4$.

(b) First, $\es(C_n)\geq 3$ for $n\geq 3$. The pattern $``1~2~3"$  gives an $S$-packing  3-coloring of $C_3$ and the pattern $``1~2~1~3~1~2~1~3\ldots1~2~1~3"$ gives an $S$-packing  3-coloring of $C_n$ when $n\equiv 0 \bmod{4}$. Thus $\es(C_n)=3$ when $n=3$ or $n \equiv 0 \bmod{4}$.

Next, if $n\ge4$ and $n\not\equiv 0 \bmod{4}$, then since $(1,3,3)\succeq(1,2,3)$,  we have $\es(C_n)\ge4$ by Proposition~\ref{spc7}. The pattern $``1~2~1~3~1~2~1~3 \ldots 1~2~1~3~4"$ gives an $S$-packing  4-coloring of $C_n$ when $n\equiv 1 \bmod{4}$ and the pattern $``1~2~1~3~1~2~1~3\ldots1~2~1~3~1~4"$ gives an $S$-packing 4-coloring of $C_n$ when $n\equiv 2 \bmod{4}$. Thus $\es(C_n)=4$ when $n\equiv1,2  \bmod{4}$.

Consider now the case $n\equiv 3 \bmod{4}$. When $n= 4k + 3$, $n\ge 7$, and $s_4<\lfloor n/2 \rfloor$, we give an $S$-packing 4-coloring $\varphi$ of $C_n = v_0v_1\ldots v_{n-1}v_0$ as:
$$ \varphi(v_i) =
\left\{\begin{array}{ll}
1; & (i\equiv 0 \bmod{4})~\text{or}~(i\equiv 2 \bmod{4}~\text{and}~i\neq 4k+2), \\
2; & (i\equiv 3 \bmod{4}~\text{and}~i<2k+1)~\text{or}~ (i \equiv 1\bmod{4}~\text{and}~i>2k+1),\\
3; & (i\equiv 1 \bmod{4}~\text{and}~i<2k+1)~\text{or}~ (i\equiv 3\bmod{4}~\text{and}~i>2k+1),\\
4; & i\in \{2k+1, 4k+2\}.
\end{array}
\right.$$
Hence $\es(C_n)=4$ when $n\equiv 3\bmod{4}$, $n\ge7$, and $s_4<\lfloor n/2\rfloor$.

When $n= 4k + 3$, $n\ge7$, and $s_4\ge\lfloor n/2 \rfloor$, the pattern ``1 2 1 3 1 2 1 3\ldots 1 2 1 3 1 4 5" is an $S$-packing 5-coloring of $C_n$. Hence  $4\le\es(C_n)\le 5$. Now suppose that there is an $S$-packing 4-coloring $\varphi$ of $C_n$. Since $s_4\ge \lfloor n/2 \rfloor$, we have $|\varphi^{-1}(4)|=1$. Without loss of generality we may assume that $\varphi(v_0)=4$.
We claim that for any edge in $C_n$ which is not incident with $v_{0}$, one of its endpoints  receives color $1$. Suppose on the contrary that there is an edge $v_iv_{i+1}\in E(C_n)$ such that $\{\varphi(v_i), \varphi(v_{i+1})\} = \{2,3\}$, where $1\le i\le n-2$. Since $n\ge7$, one of $v_{i-2}$ and $v_{i+3}$ (indices taken modulo $n$) cannot be colored under $\varphi$. This contradiction proves the claim. Since $s_2=s_3=3$, we only need to consider two cases: $\varphi(v_1)=2$ and  $\varphi(v_1)=1$.
If  $\varphi(v_1)=2$, then the colors of $v_0,v_1,\ldots,v_{4k+2}$ under $\varphi$ can be described as the pattern ``4 2 1 3 1 2 1 3 1\ldots 2 1 3 1 2 1". We have $\varphi(v_1)=\varphi(v_{4k+1})=2$ with $d_{C_n}(v_1,v_{4k+1})=3\le s_2$, a contradiction. If  $\varphi(v_1)=1$, then we may without loss of generality assume $\varphi(v_2)=2$. Then the colors of $v_0,v_1,\ldots,v_{4k+2}$ under $\varphi$ can be described as the pattern ``4 1 2 1 3 1 2 1 3\ldots 1 2 1 3 1 2''.
However, we have $\varphi(v_2)=\varphi(v_{4k+2})=2$ with $d_{C_n}(v_2,v_{4k+2})=3\le s_2$, a contradiction. Therefore $\es(C_n)=5$ when $n\equiv 3\bmod{4}$, $n\geq 7$, and $s_4\ge\lfloor n/2 \rfloor$.

If $n\not\equiv 0\bmod 4$, then $C_n$ is $\es$-vertex-critical because $\es(P_n)\le 3$ and $\es(C_n)\ge 4$.
\end{proof}

Let $G_{2k}$, $k\ge 3$, be the graph obtained from the path $P_{2k}$ by attaching a pendent vertex to each of the two support vertices of $P_{2k}$. Equivalently, $G_{2k}$ is obtained from $P_{2k-2}$ by attaching two pendant vertices to each of the two leaves of $P_{2k-2}$.

\begin{lem}\label{spc3}
If $S\in \d133$ and $k\ge 3$, then $G_{2k}$ is $4$-$\es$-vertex-critical.
\end{lem}

\begin{proof}
Let $P_{2k}=v_1v_2\ldots v_{2k}$, and let $v'_2$ and $v'_{2k-1}$ be the pendent vertices attached to $v_2$ and $v_{2k-1}$,
respectively. Coloring the vertices of $P_{2k}$ with the pattern ``1 2 1 3 1 2 1 3\ldots" and the vertices $v'_2$ and $v'_{2k-1}$ with 1 and 4, respectively, we get $\es(G_{2k})\le 4$.

Suppose now that $G_{2k}$ admits an $S$-packing 3-coloring $\varphi$. Observe that $\varphi(v_2)\in \{2,3\}$, without loss of generality assume that $\varphi(v_2)=2$. Then we have $\varphi(v_1)=1$ and $\varphi(v_3)=1$, for otherwise $\varphi(v_3)=\varphi(v_4)=1$ or $\varphi(v_4)=\varphi(v_5)=1$. If $2\leq i\leq 2k-2$, then at least one of $v_i$ and $v_{i+1}$ must be colored 1. Indeed, if we would have $\varphi(v_i)=2$ and $\varphi(v_{i+1})=3$, then $v_{i-2}$ or $v_{i+3}$ can not be colored under $\varphi$. Thus we have $\varphi(v_{2k-2})=2$ and $\varphi(v_{2k-1})=1$, or $\varphi(v_{2k-2})=3$ and $\varphi(v_{2k-1})=1$. However, this implies that $v'_{2k-1}$ or $v_{2k}$ can not be colored under $\varphi$,  a contradiction. Hence, $\es(G_{2k})=4$.

If $v\in G_{2k}$, then an $S$-packing 3-coloring of $G_{2k}-v$ can be given by coloring a longest path of each component of $G_{2k}-v$ with either the pattern $``1~2~1~3~\ldots"$ or the pattern $``2~1~3~1~\ldots"$ and coloring the pendent vertices with 1. Therefore $G_{2k}$ is 4-$\es$-vertex-critical.\end{proof}

\begin{lem}\label{HH145}
If $S\in \d133$, then the graphs $H_{14}$ and $H_{15}$ are $4$-$\es$-vertex-critical.
\end{lem}

\begin{proof}
Since $H_{14}$ and $H_{15}$ are 4-$\cs$-vertex-critical, where $S'\in \c134$,  and $(1,3,4)\succeq (1,3,3)$,  by Corollary~\ref{cor:S=134} it suffices to show that $\es(H_{14}) =\es(H_{15}) = 4$. The pattern ``4 1 2 3 1 1 1"  is an $S$-packing 4-coloring of $H_{14}$, so $\es(H_{14})\leq 4$.  Suppose that $H_{14}$ admits an $S$-packing 3-coloring $\varphi$. Then  $\{\varphi(c),\varphi(d)\}=\{2,3\}$, and so the vertex $a$ cannot be colored under $\varphi$. It follows that $\es(H_{14})=4$. The pattern ``4 1 2 1 3 1 1" is an $S$-packing 4-coloring of $H_{15}$. Moreover, since $H_{14}\subseteq H_{15}$, by Lemma~\ref{lem:subgraph}, we conclude that $\es(H_{15})=4$.
\end{proof}

Our next result, Theorem~\ref{cnnn}, follows from the following lemma and theorem.

\begin{lem}\label{GCn}
Let $S\in \d133$ and let $n\not\equiv 0\bmod{4}$, $n>3$. If a graph $G$ contains a cycle $C_n$ and $V(G)-V(C_n)\neq \emptyset$, then $G$ is not $4$-$\es$-vertex-critical.
\end{lem}

\begin{proof}
Since $V(G)-V(C_n)\neq\emptyset$, there exists a vertex $x\in V(G)$ such that $C_n\subseteq G-x$. By Lemma~\ref{spc4}, we have $\es(G-x)\ge \es(C_n)\ge 4$, and so $G$ is not 4-$\es$-vertex-critical.
\end{proof}

\begin{thm}\label{CCn56}{\rm \cite[Theorem 4.3]{klavzar-2019}}
If $G$ is a graph that contains a cycle of length $n\ge5$, where $n\not\equiv 0\bmod4$, then $G$ is $4$-$\chi_{\rho}$-vertex-critical if and only if one of the following holds.
\begin{itemize}
\item[$\bullet$] $n=5$ and $G\in \{C_5\}\cup \mathcal{C}_5$,
\item[$\bullet$] $n=6$ and $G\in  \{C_6\}\cup \mathcal{C}_6$,
\item[$\bullet$] $n\ge7$ and $G$ is isomorphic to $C_n$.
\end{itemize}
\end{thm}

\begin{thm}\label{cnnn}
Let $S\in \d133$. If $G$ is a graph that contains a cycle of length $n\ge5$, where $n\not\equiv 0 \bmod{4}$, then $G$ is $4$-$\es$-vertex-critical if and only if one of the following holds.
\begin{itemize}
\item[$\bullet$] $n = 5$ and $G\in \{C_5\}\cup \mathcal{C}_5$,
\item[$\bullet$] $n=6$ and $G\in \{C_6\}\cup \mathcal{C}_6$,
\item[$\bullet$] $n\ge7$ and $G= C_n$ except when $n\equiv 3\bmod{4}$ and $s_4\ge\lfloor n/2 \rfloor$.
\end{itemize}
\end{thm}

In order to characterize $4$-$\es$-vertex-critical graphs, where $S\in \d133$, we need to distinguish whether there are two vertices of degree at least $3$ that are at odd distance. For this sake we need the following classes of cycles that depend on a positive integer $s_4$ (this $s_4$ will, of course, be the fourth component of a packing sequence $S$):
$$\mathcal{C}_{s_4} = \{C_n, n \ge 5:\  (n\equiv 1,2 \bmod{4}) ~\text{or}~ (n\equiv3 \bmod{4}~\text{and}~s_4<\lfloor n/2 \rfloor\}\,.$$

\begin{thm}\label{h2h4}
Let $S\in \d133$ and let $G$ be a $4$-$\es$-vertex-critical graph. If all the vertices of $G$ of degree at least $3$ are pairwise at even distances in $G$, then $G\in \{H_2, H_4\} \cup \mathcal{C}_{s_4}$.
\end{thm}

\begin{proof}
By Lemmas~\ref{pcn} and~\ref{KCH}, every graph from $\{H_2, H_4\} \cup \mathcal{C}_{s_4}$ is $4$-$\es$-vertex-critical.  If $\Delta(G)\le2$, then $G \in \{P_n, C_n\}$, hence $G\in \mathcal{C}_{s_4}$. Suppose now that $\Delta(G)\ge 3$ and that all the vertices of degree at least $3$ are pairwise at even distances in $G$. Let $u\in V(G)$ be an arbitrary vertex of degree at least $3$. Then define $\varphi: V(G)\rightarrow[3]$ by:
$$ \varphi(v) =
\left\{\begin{array}{ll}
1; & d_G(u,v)\equiv 1,3 \bmod{4}, \\
2; & d_G(u,v)\equiv 0\bmod{4},\\
3; & d_G(u,v)\equiv 2\bmod{4}.
\end{array}
\right.$$
By Lemma~\ref{spc5}, $G$ is a connected graph, and so $\varphi$ is well-defined.  Since $G$ is 4-$\es$-vertex-critical,  there are two vertices $x,y\in V(G)\setminus \{u\}$ such that $\varphi(x)=\varphi(y)=i$ and $d_G(x,y)\le s_i$ for some $i\in [3]$. Let $P$ and $P'$ be arbitrary shortest $u,x$-path and $u,y$-path in $G$, respectively. Let $w\in V(P)\cap V(P')$ such that $d_G(u,w)$ is as large as possible. Then we have $w\neq x,y$ and $d_G(u,x)=d_G(u,y)$. If $w=x$ or $d_G(u,x)<d_G(u,y)$, then $d_G(u,y)\le d_G(u,x)+s_i<d_G(u,x)+4$, and so $\varphi(x)\neq \varphi(y)$ by the definition of $\varphi$, which leads to a contradiction. Thus we have $d_G(w)\ge3$, and so $\varphi(w)\in \{2,3\}$.

\medskip\noindent
{\bf Claim.} $G$ contains a cycle consisting of  $wPx$, $wP'y$, and a shortest $x,y$-path.

\medskip\noindent
{\bf Proof.}
Let $P''$ be a shortest $x,y$-path. By the choice of $w$, it suffices to show that $(V(P'')\cap V(P))\setminus\{x\} = (V(P'')\cap V(P'))\setminus\{y\}=\emptyset$.

Suppose that $(V(P'')\cap V(P))\setminus\{x\}\neq \emptyset$. Note that this can only happen when $\varphi(x)=\varphi(y)\in \{2,3\}$ and $d_G(x,y)\in \{2,3\}$. Without loss of generality, let $\varphi(x)=\varphi(y)=3$. If $d_G(x,y)=2$, let $P''=xzy$. Then $d_G(z)\geq 3$ and $d_G(u,z)=d_G(u,x)-1\equiv 1\bmod{4}$, contradicting the fact that the vertices of degree at least $3$ are at even distance. For $d_G(x,y)=3$, let $P''=xz_1z_2y$. Then $xz_2\notin E(G)$. If $z_2\in V(P'')\cap V(P)$, then $d_G(u,y)\le d_G(u,z_2)+d_G(z_2,y)=d_G(u,x)-2+1<d_G(u,x)$, contradicting the fact that $d_G(u,x)=d_G(u,y)$. Therefore $z_2\notin V(P'')\cap V(P)$. However, if $z_1\in V(P'')\cap V(P)$, then $d_G(z_1)\ge 3$ and $d_G(u,z_1)=d_G(u,x)-1\equiv 1\bmod4$, a contradiction. Therefore $G$ contains a cycle, say $C_n$, consisting of $wPx$, $wP'y$ and a shortest $x,y$-path, and so $n=2d_G(u,x)-2d_G(u,w)+d_G(x,y)$. \qed

\medskip
Suppose $\varphi(x)=\varphi(y)=1$ and $d_G(x,y)=1$. Then we have  $n\equiv 3\bmod{4}$. If $n\equiv 3\bmod{4}$ and $n>3$, then $G = C_n$ with $s_4<\lfloor n/2\rfloor$ by Theorem~\ref{cnnn}. If $n=3$, then $wxy$ is a triangle with $d_G(w)\ge 3$ and $d_G(x)=d_G(y)=2$. If $N^2_G(w)\neq\emptyset$, then $G = H_4$. If $N^2_G(w)=\emptyset$ and $N_G(w)\setminus\{x,y\}$ is an independent set, then a mapping $\varphi$ with $\varphi(N_G(w)\setminus\{x,y\})=\varphi(x)=1$, $\varphi(w)=2$ and $\varphi(y)=3$ is an $S$-packing 3-coloring of $G$. Moreover, it is easy to see that no edge can be added to $H_2$ and to $H_4$. Hence $G\in \{H_2, H_4\}$ when $n=3$. Suppose $\varphi(x)=\varphi(y)=2$ or $\varphi(x)=\varphi(y)=3$. If $d_G(x,y)=i$, $i\in [3]$, then $n \equiv i\bmod{4} \ge 5$ because $d_G(u,x)$ and $d_G(u,w)$ are even. Therefore $G\in \mathcal{C}_{s_4}$ by Theorem~\ref{cnnn}.
\end{proof}

\begin{thm}\label{aod}
Let $S\in \d133$ and let $G$ be a $4$-$\es$-vertex-critical graph in which there exist two vertices of degree at least $3$ that are at odd distance. Then
$$G\in \{K_4\} \cup {\cal H}_{1, 3, 5, 7, 14, 15}\cup \{G_{2k}:\ k\ge3 \}\cup \mathcal{C}_5\cup \mathcal{C}_6\,.$$
\end{thm}

\begin{proof}
Let $u,v\in V(G)$ with $d_G(u),d_G(v)\ge3$ such that $d_G(u,v)=\ell$ is odd and as small as possible. We consider the following two cases.

\medskip\noindent
{\bf Case 1}: $\ell\ge3$. \\
By the choice of $u$ and $v$, $N_G(u)\cap N_G(v)=\emptyset$ and each vertex on a shortest $u,v$-path has degree $2$ in $G$.  Therefore $G_{\ell+3}= G_{2k}$ is a spanning subgraph of $G$, where $k = \frac{\ell+3}{2}$  by Lemma~\ref{spc3}.  Let the vertices of $G_{2k}$ be denoted as in Lemma~\ref{spc3} with $u=v_2$ and $v=v_{2k-1}$. By symmetry, only some of the edges $v_1v'_{2}$, $v_1v_{2k}$, and $v_1v_{2k-1}$ can possibly be added to $G_{2k}$. If $v_1v'_{2}\in E(G)$, then $H_4\subseteq G-v_{2k}$, a contradiction. If $v_1v_{2k}\in E(G)$, then $C_{2k}\subseteq G$. Further, we have $2k\equiv  0\bmod4\ge8$, for otherwise $\chi_{S}(C_{2k})=4$
by Theorem~\ref{cnnn} and $C_{2k}\subseteq G-v'_2$. Then we can find a copy of $G_6$ consisting of $v_3v_2v_1v_{2k}v_{2k-1}v_{2k-2}$ and the pendent vertices $v'_{2}$ and $v_{2k-1}'$  contained in $G-v_4$, which also leads to a contradiction. If $v_1v_{2k-1}\in E(G)$, then there is a cycle $C_{2k-1}\subseteq G-v'_{2}$ with $2k-1\not\equiv0\pmod4>3$, again a contradiction. We conclude that in Case 1, $G = G_{2k}$.

\medskip\noindent
{\bf Case 2}: $\ell=1$. \\
We claim that in this case, $G\in \{K_4\} \cup {\cal H}_{1, 3, 5, 7, 14, 15}\cup \mathcal{C}_5\cup \mathcal{C}_6$.

If $G$ contains a cycle $C$ from $\mathcal{C}_{s_4}$, then by Theorem~\ref{cnnn}, $C$ is a spanning subgraph of $G$. Since there are two vertices of degree at least 3 which are of distance 1 in $G$, we have $G\in \mathcal{C}_5\cup \mathcal{C}_6$. Suppose $G$ contains $H_4$ as a subgraph. Then $H_4$ is a spanning subgraph of $G$. Since $G$ has two vertices of degree at least 3, $G\in \mathcal{C}_5\cup \{H_1\}$. By the same argument, if $G$ contains $H_2$ as a spanning subgraph, then $G\in \mathcal{C}_5$. Therefore we may assume $G$ does not contain a graph from $\mathcal{C}_{s_4}\cup \{H_2,H_4\}$ as a subgraph. Let $a=u$ and $c=v$.

Suppose that $|N_G(a)\cap N_G(c)|\ge 2$. If $d_G(x)=2$ for any $x\in N_G(a)\cap N_G(c)$, then $V(G)\setminus\{a,c\}$ is an independent set in $G$  since $G$ does not contain $H_2$ or $H_4$ as a subgraph, and so a coloring $\varphi$ of $G$ with $\varphi(V(G)\setminus\{a,c\})=1$, $\varphi(a)=2$ and $\varphi(c)=3$ is an $S$-packing 3-coloring, a contradiction. Then either $bd \in E(G)$ for some $b,d\in N_G(a)\cap N_G(c)$ and so $G= K_4$, or there is a vertex $x\in N_G(a)\cap N_G(c)$ such that $N_G(x)\setminus (\{a,c\}\cup V(N_G(a)\cap N_G(c)))\neq\emptyset$ and so $H_1\subseteq G$. Since $G$ contains no cycle from $\mathcal{C}_{s_4}$, we infer that no more edges can be added to $H_1$. Hence $G = H_1$.

Suppose that $|N_G(a)\cap N_G(c)|=1$, and let $b\in N_G(a)\cap N_G(c)$. If $d_G(b)=2$, then since $G$ does not contain $H_2$ and $H_4$ as a subgraph,  a coloring $\varphi$ of $G$ with  $\varphi(V(G)\setminus\{a,c\})=1$, $\varphi(a)=2$, and $\varphi(c)=3$ is an $S$-packing 3-coloring, a contradiction. Therefore, $d_G(b)\ge3$.
If $|N_G(b)\cap N_G(a)|\ge2$ or $|N_G(b)\cap N_G(c)|\ge2$, then $G\in \{K_4,H_1\}$ because $d_G(a)\ge3$, $d_G(c)\ge3$, and $ac\in E(G)$. If $N_G(b)\cap N_G(a)=c$ and $N_G(b)\cap N_G(c)=a$, then $H_3\subseteq G$  since $d_G(z)\ge 3$ for $z\in \{a,b,c\}$. Let $d,e,f$ be the three vertices of $H_3$ as shown in Fig.~\ref{spcfig2}. If $af\in E(G)$, then $H_1\subseteq G-d$, a contradiction. If $df\in E(G)$, then $C_5\subseteq G-e$, again a contradiction. Therefore $G= H_3$.

Lastly, consider the case when $N_G(a)\cap N_G(c)=\emptyset$.  If $d_G(w)=1$ for each $w\in (N_G(a)\cup N_G(c))\setminus \{a,c\}$, then a mapping $\varphi$ with $\varphi(V(G)\setminus \{a,c\})=1$, $\varphi(a)=2$, and $\varphi(c)=3$ is an $S$-packing 3-coloring of $G$, contradicting the fact that $\es(G)=4$. Let $x_1\neq x_2\in N_G(a)\setminus\{c\}$ and $y_1\ne y_2\in N_G(c)\setminus\{a\}$. Without loss of generality assume that $N_G(x_1)\setminus \{a\}\neq \emptyset$.  If $x_1x_2\in E(G)$, then $H_4\subseteq G-y_2$, a contradiction. Hence $x_1x_2\notin E(G)$ and $y_1y_2\notin E(G)$ by symmetry.
If $x_1y_1\in E(G)$, then $H_5\subseteq G$. Moreover, $H_5$ is a spanning subgraph of $G$ by Lemma~\ref{KCH}. If $x_1y_1\in E(G)$ and $x_2y_2\in E(G)$, then $C_6$ is a proper subgraph of $G$, again a contradiction. If $x_1y_1\in E(G)$ and only one of $x_1y_2$ and $x_2y_1$ is contained in $G$, then $H_7\subseteq G$. Since there is no more edge which can be added to $H_{7}$, we get $G\in {\cal H}_{5,7}$ when $x_1y_1\in E(G)$.  If there is a vertex $x_1'\in N(x_1)\setminus \{a, x_2,y_1,y_2\}$, then $H_{14}$ is a spanning subgraph of $G$ by Lemma~\ref{HH145}.  If one of the edges $x_1'a$, $x_1x_2$, and $y_1y_2$ is contained in $G$, there is a vertex $y\in V(G)$ such that $H_4\subseteq G-y$, a contradiction. If one of the edges $x_1'c$, $x_1y_1$, $x_1y_2$, $y_1x_2$, and $x_2y_2$ is contained in $G$, there is a vertex $y\in V(G)$ such that $H_5\subseteq G-y$, a contradiction. If $x_1'y_i\in E(G)$, then $C_5\subseteq G-x_2$, a contradiction.  Since  $ay_i,cx_i\notin E(G)$ for $i\in[2]$, only  $x_1'x_2$ can be  possibly contained in $G$,  and $H_{15}\subseteq G$ when  $x_1'x_2 \in E(G)$. Moreover, since $H_{15}$ is $4$-$\es$-vertex-critical, and there is no more edge can be added to $G$, we conclude that $G\in {\cal H}_{14, 15}$ when $N(x_1)\setminus \{a, x_2,y_1,y_2\}\neq \emptyset$.
\end{proof}

Theorems~\ref{aod} and~\ref{h2h4} are combined into the following final result of this paper.

\begin{cor}\label{spccor1}
Let $S\in \d133$. Then a graph $G$ is $4$-$\es$-vertex-critical if and only if
$$G\in \{K_4\} \cup {\cal H}_{1-5, 7, 14, 15}
\cup \mathcal{C}_5 \cup \mathcal{C}_6
\cup \mathcal{C}_{s_4} \cup \{G_{2k}:\ k\ge 3\}\,.$$
\end{cor}

\newpage
\section*{Acknowledgements}

Hui Lei was partially supported by the National
Natural Science Foundation of China (No. 12001296) and the Fundamental Research Funds
for the Central Universities, Nankai University (Nos. 63201163 and 63211093). Xiaopan Lian and Yongtang Shi were partially supported by the National Natural Science Foundation of China (Nos. 12161141006 and
12111540249), the Natural Science Foundation of Tianjin (Nos. 20JCJQJC00090 and 20JCZD-JC00840). Sandi Klav\v{z}ar acknowledges the financial support from the Slovenian Research Agency (research core funding P1-0297, and projects N1-0285, J1-2452).

\end{document}